\documentclass[12pt, a4paper]{amsart} 

\usepackage{mathptmx} 
\usepackage{amsfonts,amsthm,amssymb,dsfont,stmaryrd,mathrsfs} 
\usepackage[leqno]{amsmath} 
\usepackage{array}
\usepackage{colortbl}
\usepackage{comment} 
\usepackage{color}

\usepackage{graphicx,caption} 
\usepackage[pdfstartview=FitH]{hyperref} 
\usepackage{pdfpages} 
\usepackage{longtable,supertabular, multirow} 
\usepackage{doi}
\usepackage{tikz, tikz-cd}
\usetikzlibrary{matrix,arrows,decorations.pathmorphing}

\usepackage[hmarginratio={1:1}, vmarginratio={1:1}, top=3cm, bottom=2cm, left=3cm, right=3cm]{geometry} 

\usepackage{setspace}
\usepackage{mathtools} 

\theoremstyle{plain}
\newtheorem*{thm*}{Theorem} 
\newtheorem{thm}{Theorem} [section]
\newtheorem*{thm2*}{Theorem \ref{thmDecNilp}}
\newtheorem*{thm3*}{Theorem \ref{thmDecNilp2}}
\newtheorem{thm1}{Theorem}

\newtheorem{thmA}{Theorem}
\newtheorem{lem}[thm]{Lemma}
\newtheorem*{lem*}{Lemma}
\newtheorem{prop}[thm]{Proposition}
\newtheorem*{prop*}{Proposition}

\newtheorem{cor}[thm]{Corollary}
\newtheorem*{cor*}{Corollary}
\newtheorem{corA}[thmA]{Corollary}

\theoremstyle{definition}

\newtheorem*{defn*}{Definition}

\newtheorem{conjectureA}{Conjecture}
\newtheorem{conjectureAp}{Conjecture}
\newtheorem*{conjecture*}{Conjecture}
\newtheorem{exmp}[thm]{Example}
\newtheorem*{exmp*}{Example}

\newtheorem*{prob*}{Problem}

\newtheorem*{ques*}{Question}
\newtheorem{remk}[thm]{Remark}
\newtheorem*{remk*}{Remark}
\newtheorem*{obsv*}{Observation}

\def\nn{\mathbb{N}}
\def\zz{\mathbb{Z}}

\def\cc{\mathbb{C}}
\def\ff{\mathbb{F}}

\def\gcd{{\rm gcd}}

\def\ker{{\rm ker}}

\renewcommand{\c}{\chi}
\renewcommand{\d}{\delta}

\newcommand{\e}{\epsilon}

\newcommand{\om}{\omega}

\newcommand{\gfk}{{\mathfrak{g}}}

\newcommand{\Ccal}{{\mathcal C}}

\newcommand{\Ocal}{{\mathcal O}}

\newcommand{\CC}{{\mathbb{C}}}

  {\end{list}}

\begin{document}

\title{Certain Diagonal Equations and Conflict-Avoiding Codes of Prime Lengths}

\author{Liang-Chung Hsia}
\address{Department of Mathematics, National Taiwan Normal University, Taipei 11677, Taiwan, ROC}
\email{hsia@math.ntnu.edu.tw}

\author{Hua-Chieh Li}
\address{Department of Mathematics, National Taiwan Normal University, Taipei 11677, Taiwan, ROC}
\email{li@math.ntnu.edu.tw}

\author{Wei-Liang Sun}
\address{Department of Mathematics, National Taiwan Normal University, Taipei 11677, Taiwan, ROC}
\email{wlsun@ntnu.edu.tw}

\maketitle

\begin{abstract}
We study the construction of optimal conflict-avoiding codes (CAC) from a number theoretical point of view. The determination of the size of optimal CAC of prime length $p$ and weight 3 is formulated in terms of the solvability of certain twisted Fermat equations of the form $g^2 X^{\ell} + g Y^{\ell} + 1 = 0$  
over the finite field $\ff_{p}$ for some primitive root $g$ modulo $p.$ We treat the problem of solving the twisted Fermat equations in a more general situation by allowing the base field to be any finite extension field $\ff_q$ of $\ff_{p}.$ We show that for $q$ greater than a lower bound of the order of magnitude  $O(\ell^2)$ there exists a generator $g$ of $\ff_{q}^{\times}$ such that the equation in question is solvable over $\ff_{q}.$ Using our results
we are able to contribute new results to the construction of optimal CAC of prime lengths and weight $3.$ \\


\end{abstract}


\section{Introduction}
\label{sec1}

A binary protocol sequence set for transmitting data packets over a multiple-access collision channel without feedback is called a {\it conflict-avoiding code} (CAC) in information theory. It has been studied a few decades ago by \cite{Mat, NguGyoMas, GyoVaj, TsyRub, LevTon, Lev}. A mathematical model for CACs of {\it length} $n$ and ({\it Hamming}) {\it weight} $w$ is as follows. Let $\zz / n \zz$ be the additive group of the integer ring $\zz$ modulo $n$. For a $w$-subset ${\bf x} = \{x_1, \ldots, x_w\}$ of $\zz / n \zz$, let $\Delta({\bf x}) = \{x_i - x_j \mid i \neq j\}$. A CAC of length $n$ and weight $w$ is a collection $\mathcal{C}$ of $w$-subsets of $\zz / n \zz$ such that $\Delta({\bf x}) \cap \Delta({\bf y}) = \varnothing$ for every distinct ${\bf x}, {\bf y} \in \mathcal{C}$. Each $w$-subset ${\bf x}$ of $\mathcal{C}$ is called a {\it codeword}. A CAC $\mathcal{C}$ is said to be {\it optimal} if its size is maximal among all CACs of the same length and weight. In the case where the weight is one or two, there is no difficulty to find the optimal size. However, for weights more than $2$, finding an optimal CAC and determining its size is still an open problem. The first challenge is the case of weight $3$. 
One of the purpose of this note is to treat the  problem of finding optimal CAC's  from a number theoretical point of view and 
contribute new results to the construction of optimal CAC of weight $3$. Thus, all CACs which we are concerned with will be of weight $3$. 

In the case of even lengths and weight 3, the problem of constructing optimal CACs has a complete answer by the work~\cite{LevTon, JimMisJanTeyTon, MisFuUru, FuLinMis}. In contrast, it is incomplete for odd lengths. Let $o_m(2)$ denote the multiplicative order of $2$ modulo a positive odd integer $m$. For a CAC of odd length $n$, we write $n = n_1 n_2$ such that $o_p(2)$ is a multiple of $4$ for all the  prime divisors $p$ of $n_1$ while $o_p(2)$  is not divisible by $4$  
for all prime divisors $p$ of $n_2.$ An optimal CAC are constructed in \cite{LevTon, Lev} when $n_2 = 1$. If $n_2 \neq 1$, it is proved in \cite{FuLoShu} 
that an optimal CAC of length $n$ can be constructed from an optimal CAC of length $n_2$. It is also proved in \cite{FuLoShu} that an optimal CAC of a 
prime power length can be constructed if we know how to construct an optimal CAC of prime length provided that the prime $p$ in question is a 
non-Wieferich prime. For other odd lengths or {\it tight/equi-difference} CACs, we refer to \cite{Mom, WuFu, LinMisSatJim, MisMom, HsiLiSun2} for the constructions. It turns 
out that  CACs of prime lengths are the fundamental cases needed to be constructed. This naturally leads us to study CACs of {\em prime lengths} and 
weight $3$. 

Let $p$ be an odd prime and denote by $\ff_p = \zz / p \zz$  a finite field of $p$ elements. Recall that a codeword of the form $\{0, a, 2a\}$ is said to 
be {\it equi-difference}. In the paper~\cite{LevTon}, the authors show that there exists an optimal CAC  consisting of  $\frac{p-1}{4}$ equi-difference codewords in the case where  $4 \mid o_p(2)$. 
In contrast,  if $4 \nmid o_p(2)$ then a  CAC consisting of equi-difference codewords only is usually not optimal. 
By analyzing  nonequi-difference codewords, an upper bound of the size of optimal CAC of odd length is given in~\cite{FuLoShu}. 
Let us recall their results for the case of CAC with prime lengths. 
%
%
Put 
$$
\Ocal(p) = 
\left\{ \begin{array}{cl}
\frac{p-1}{2 o_p(2)} & \text{if $o_p(2)$ is odd,} \\
\frac{p-1}{o_p(2)} & \text{if $o_p(2) \equiv 2 \pmod{4}$,} \\
0 & \text{if $4 \mid o_p(2)$} 
\end{array} \right.
$$
and $M(p)$ to be the size of optimal CAC of length $p$. Then, by~\cite[Lemma~3]{FuLoShu} one has 
\begin{align} 
\label{upper bound: FLS} 
\frac{p - 1 - 2 \Ocal(p)}{4} \leq M(p) \leq \frac{p - 1 - 2 \Ocal(p)}{4} + \left\lfloor \frac{\Ocal(p)}{3} \right\rfloor.
\end{align}
Note that in the case where $\Ocal(p)\le 2$, inequality~\eqref{upper bound: FLS} already gives that $M(p) = \frac{p - 1 - 2 \Ocal(p)}{4}.$ 
For $\Ocal(p)\ge 3$ the authors provide an algorithm for constructing nonequi-difference CAC and conjectured that the algorithm produces a CAC consists of 
$\frac{p - 1 - 2 \Ocal(p)}{4}$ equi-difference codewords and $\left\lfloor \frac{\Ocal(p)}{3} \right\rfloor$ nonequi-difference codewords. 
In other words, the upper bound in~\eqref{upper bound: FLS} can be attained and hence the CAC obtained by their algorithm is actually an 
optimal CAC. The key property needed for their algorithm to work is given as Conjecture~\ref{conjA} below. 
For our purpose, we rephrase their conjecture in terms of cosets of the subgroup generated by $-1$ and $2$ in the multiplicative group  
$\ff_p^{\times} = \ff_p \setminus \{0\}$  of $\ff_{p}$. 
\begin{conjectureA}[{\cite[Conjecture~1]{FuLoShu}}]
\label{conjA}
Let $p$ be a non-Wieferich prime. Then there are $3 \left\lfloor \frac{\Ocal(p)}{3} \right \rfloor$ cosets  
$A_1, B_1, C_1, \ldots, 
A_{\left\lfloor \frac{\Ocal(p)}{3} \right \rfloor}, 
B_{\left\lfloor \frac{\Ocal(p)}{3} \right \rfloor}, 
C_{\left\lfloor \frac{\Ocal(p)}{3} \right \rfloor}
$ 
of the subgroup generated by $-1$ and $2$ in 
$\ff_p^{\times}$ such that for each $i=1, \ldots, \left\lfloor \frac{\Ocal(p)}{3} \right \rfloor$ 
there exists a triple $(a_{i}, b_{i}, c_{i}) \in A_i \times B_i \times C_i$ satisfying 
$$
a_{i} + b_{i} + c_{i} = 0 \quad \text{in } \ff_p.
$$
\end{conjectureA}
Throughout this article we denote by
$H = \langle -1, 2 \rangle$  the subgroup generated by $-1$ and $2$ in $\ff_p^{\times}$ 
and set 
$$
\ell_0 = [\ff_p^{\times}: H]
$$ 
for the index of $H$ in $\ff_{p}^{\times}.$  Notice that if $4 \nmid o_{p}(2)$ then by definition we have that  $\Ocal(p) = \ell_{0}.$ 
Furthermore, Conjecture~\ref{conjA} is a non-empty statement if and only if $\Ocal(p)=\ell_{0} \ge 3.$ 

The idea behind Conjecture~\ref{conjA} is the following. Suppose that Conjecture~\ref{conjA} holds, then  each triple $(a_{i}, b_{i}, c_{i})$ in 
the conjecture corresponds to a nonequi-difference codeword 
${\bf x_{i}} = \{0, a_{i}, -c_{i}\}$ with difference set  $\Delta({\bf x_{i}}) = \{\pm a_{i}, \pm b_{i}, \pm c_{i}\}$.  Hence, we have 
$\left\lfloor \frac{\Ocal(p)}{3} \right \rfloor$ nonequi-difference codewords whose difference sets are disjoint. From the complement of 
$\cup_{i=1} \Delta({\bf x_{i}})$ in $\ff_{p}^{\times}$, their algorithm then produces $\frac{p - 1 - 2 \Ocal(p)}{4}$ equi-difference codewords
and hence gives a CAC of size matching the upper bound given in~\eqref{upper bound: FLS}. 

As an illustration, we briefly discuss the case treated in~\cite[Example~3]{FuLoShu} where the length $p = 31$. 
Note that $o_{31}(2) = 5$ and hence $\Ocal(31) = 3.$ Then Conjecture~\ref{conjA} predicts that there are $3\left\lfloor \frac{\Ocal(31)}{3} \right \rfloor =  3$ cosets and one element in each coset such that their sum is zero. 
One finds that the triple  $(2, 3, -5)$ gives a solution and the corresponding 
codeword is $\{0, 2, 5\}$  whose difference set is just $\{\pm 2, \pm 3, \pm 5\}$ while $2$, $3$ and $-5$  lie exactly in  three distinct cosets of $H$ in $\ff_p^{\times}$. Moreover, there are six equi-difference codewords $\{0, 4, 8\}$, $\{0, 6, 12\}$, $\{0, 7, 14\}$, $\{0, 9, 18\}$, $\{0, 10, 20\}$ and $\{0, 15, 30\}$ produced by their algorithm. In total, one concludes that the size of an optimal CAC of length $31$ is $M(31) = 7. $

Independently, in~\cite{MaZhaShe} the authors proposed a conjecture which provides solutions to  
the existence of the triples  $(A_i, B_i, C_i)$ in Conjecture~\ref{conjA}  in terms of the group structure of $\ff_p^{\times} / H$. 

\begin{conjectureA}[{\cite[Conjecture]{MaZhaShe}}]
\label{conjB}
Let $p$ be an odd prime. If $\ell_0 \geq 3$, 
then there exist  $b \in gH$ and  $c \in g^{2}H$ such that 
$$
1 + b + c = 0 \quad \text{in } \ff_p
$$ 
for some generator $g$ of  $\ff_p^{\times}.$ 
\end{conjectureA}
\begin{remk}
\label{rmk: conjB}
We see that Conjecture~\ref{conjB} implies Conjecture~\ref{conjA} by setting $A_1 = H$, $B_1 = gH$, $C_1 = g^{2}H$, $A_2 = g^{3}H$, $B_2 = g^{4}H$, $C_2 = g^{5}H, \ldots, A_{e} = g^{3e-3}H$, $B_{e} = g^{3e-2} H$ and $C_{e} = g^{3e-1} H$ where $e = \left\lfloor \frac{\ell_{0}}{3} \right \rfloor$. Moreover, Conjecture~\ref{conjB} does not assume that $4 \nmid o_p(2)$.  
\end{remk}

Note that the subgroup $H = \langle -1, 2 \rangle$ consists of all the $\ell_0$-th power of elements of $\ff_p^{\times}.$   It follows that the elements 
$b$ and $c$ in Conjecture~\ref{conjB} are of the forms $g y^{\ell_0}$ and $g^2 x^{\ell_0}$ respectively for some $x, y \in \ff_p^{\times}$. Observe 
that if $\ell_{0}\ge 3$ then  any $\ff_{p}$-rational solutions $(x,y)$ to the the {\it diagonal equation} $g^2 X^{\ell_0} + g Y^{\ell_0} + 1 = 0$ must satisfy $x y \ne 0$ since  $-1\in H$ and $g$ is a generator of $\ff_{p}^{\times}.$  Thus, any $\ff_{p}$-rational solution gives a pair of elements $b$ and $c$ in Conjecture~\ref{conjB}.  So Conjecture~\ref{conjB} is equivalent to the following statement.

\begin{conjectureAp}
\label{conjB'}
Let $p$ be an odd prime. If $\ell_0  \geq 3$, then there is a generator $g$ of $\ff_{p}^{\times}$ such that the diagonal equation \begin{align} 
\label{eqn: CAC} 
g^2 X^{\ell_0} + g Y^{\ell_0} + 1 = 0 
\end{align} 
is solvable over $\ff_p$.
\end{conjectureAp}

The formulation in Conjecture~\ref{conjB'} has the advantage that the number of $\ff_{p}$-rational solutions to Equation~\eqref{eqn: CAC} can be 
computed in terms of certain character sums which have been well studied in number theory.  By establishing valid cases in Conjecture~\ref{conjB'}, we also obtain the cases where Conjecture~\ref{conjB} as well as Conjecture~\ref{conjA} are true. Therefore, by studying the solvability of Equation~\eqref{eqn: CAC} over $\ff_{p}$, we are able to provide new results to the construction of optimal CACs.

Motivated by Conjecture~\ref{conjB'}, instead of working on the diagonal equations as~\eqref{eqn: CAC} over the prime field $\ff_{p}$ and the specific exponent $\ell_{0}$, we will look at  general situations by taking the base field to be a finite extension of $\ff_{p}$ and the exponent in the equation is allowed to be more general  than $\ell_{0}$. 
Let $q$ be a prime power and $\ell$ be a proper divisor of $q-1$. We consider the solvability of the following diagonal equation 
\begin{equation}
\label{eqn: weil} 
g^2 X^{\ell} + g Y^{\ell} + 1 = 0
\end{equation}
over a finite field $\ff_q$ of $q$ elements, where $g$ is a generator of  the multiplicative group $\ff_q^{\times}$ of $\ff_{q}.$ 
In view of Conjecture~\ref{conjB'}, we're interested in whether or not there exists a generator $g$ such that Equation~\eqref{eqn: weil}  has a $\ff_{q}$-rational solution. However,  the answer can be false for divisors of $q-1$ other than $\ell_{0}.$ For example,  
in the case where  $(q, \ell) = (13, 6), (23, 11)$ there does not exist any generator of $\ff_{q}^{\times}$ such that~\eqref{eqn: weil} 
has a $\ff_{q}$-rational solution. 
On the other hand, as a consequence of our main result below,   Equation~\eqref{eqn: weil} does have a 
 $\ff_{q}$-rational solution for some generator $g$ of $\ff_{q}^{\times}$ provided that $q \geq 19$ if $\ell = 6$ and $q\ge 322$  if $\ell =11.$ 
Our first main result is to give a lower bound for $q$ such that  Equation~\eqref{eqn: weil} has a $\ff_{q}$-rational solution
for some generator $g$ of $\ff_q^{\times}$.

\begin{thmA}[= Theorem~\ref{thm4.1}]
\label{thm1.1}
Let $q$ be a prime power and let $\ell$ be a proper divisor of  $q-1$. If 
$$
q \geq (2^{\om(\ell)} (\ell - 3 - \delta) + 2)^2 - 2
$$ 
where $\om(\ell)$ is the number of distinct prime divisors of $\ell$ and 
$$
\delta = \left\{\begin{array}{ll} 1 & \text{if } 4 \mid \ell, \\ 0 & \text{otherwise,}\end{array} \right.
$$ 
then there is a generator $g$ of $\ff_q^{\times}$ such that Equation~\eqref{eqn: weil} is solvable over $\ff_q$.
\end{thmA}

\begin{remk}
\label{rmk: weil bound}
It follows from the Hasse-Weil bound  (see Theorem~\ref{thm:hasse-weil}) that the number of $\ff_{q}$-rational solutions to Equation~\eqref{eqn: weil} 
is bounded below by $q + 1 - 2 \gfk_{\ell}\sqrt{q}$ where $\gfk_{\ell} = (\ell-1)(\ell-2)/2$ is the genus of the curve defined by~\eqref{eqn: weil} over 
$\ff_{q}.$ As a result, Equation~\eqref{eqn: weil} has a $\ff_{q}$-rational solution for any $g\in \ff_{q}^{\times}$ provided that $q > (\ell-1)^{2}(\ell-2)^{2}.$ It is reasonable to expect that this lower bound can be improved under the weaker condition  given in Conjecture~\ref{conjB'}.  
What we have shown in Theorem~\ref{thm1.1} is that the improved lower bound has the order of magnitude  $O(\ell^{2})$. 
\end{remk}


Theorem~\ref{thm1.1} gives a sufficient condition for the truth of Conjecture~\ref{conjB'} (and so are Conjecture~\ref{conjA} and \ref{conjB}) in the case where $q = p$ a prime number and $\ell = \ell_0$.  Thus, under the given sufficient condition  an optimal CAC of length $p$ with $4 \nmid o_p(2)$ 
and weight $3$ has the desired size. 

\begin{corA}
\label{thm1.2}
Let $p$ be an odd prime such that $4 \nmid o_p(2)$, let $\ell_0 = [\ff_p^{\times} : H]$ and let $\om(\ell_{0}), \delta$ be as in \textup{Theorem~\ref{thm1.1}} with respect to $\ell_0$. If $p \geq (2^{\om(\ell_{0})} (\ell_0 - 3 - \delta) + 2)^2 - 2$, then an optimal conflict-avoiding code of length $p$ and weight $3$ has the size $$\frac{p-1-2 \ell_0}{4} + \left\lfloor \frac{\ell_0}{3} \right\rfloor.$$
\end{corA}

Applying Corollary~\ref{thm1.2},  we can establish the truth of Conjecture~\ref{conjB'} unconditionally for primes with small values of $\ell_{0}$. 
For instance, if $1\le \ell_{0}\le 6$  then  Conjecture~\ref{conjB'} is true (see,  Corollary~\ref{cor: small ell}, \ref{cor4.1} and~\ref{cor4.2}). Combining the results computed in~\cite{MaZhaShe}, Theorem~\ref{thm1.1} confirms the validity of Conjecture~\ref{conjB'} for a large range of $\ell_{0}$.  For instance, 
if $\ell_{0}$ is  prime power satisfying $\ell_{0} < 16411$ or if it has two distinct prime divisors such that 
$\ell_{0} < 8197$ then Conjecture~\ref{conjB'} is true for prime numbers $p$ with $\ell_{0}$ satisfying properties just stated (see Theorem~\ref{thm:C} and Theorem~\ref{thm:D} for more cases). 


%

 The organization of this note is as follows. In Section~\ref{sec2'} we fix some notations and discuss some well-known facts related to 
Equation~\eqref{eqn: weil}. In particular, by applying Hasse-Weil bound, we give a proof of the facts that Equation~\eqref{eqn: weil} is 
solvable over $\ff_{q}$ in the case where $1\le \ell \le 4$ (Corollary~\ref{cor: small ell}).  
Then, we collect and prove necessary results that are needed in the proof of the main result in Section~\ref{sec:sec3}. 
One of the key ingredients is {\em Ramanujan's sum} which we recall in Lemma~\ref{lem:ramanujan sum}. 
Section~\ref{sec:proof} is devoted to the proof of Theorem~\ref{thm1.1}. By appropriately organizing the character sum in the expression for 
the number of solutions to Equation~\eqref{eqn: weil}, we are able to obtain the desired bound given in Theorem~\ref{thm1.1} 
for the number of solutions. In the final section, we 
apply our main result to the problem of the size of optimal CAC and deduce a large range of $\ell_{0}$ such that Conjecture~\ref{conjB'} (as well as 
Conjecture ~\ref{conjB} and \ref{conjA}) hold.



\section{Preliminaries}
\label{sec2'}
In this section, we fix  notations and present some facts that are related to the question of solvability of Equation~\eqref{eqn: weil}.
Let $\ff_q$ be a finite field of $q$ elements where 
$q$ is a power of the prime $p.$ Fix a generator $g$ of $\ff_q^{\times}$ and a proper divisor $\ell$ of $q-1$. 
Let $L$ be the subgroup of all $\ell$-th power of elements of $\ff_q^{\times}.$ We have that $\ff_q^{\times} / L$ is generated by the coset $g L$ and 
$\ell$ is the order of the cyclic group $\ff_q^{\times} / L.$ 
 
Recall that we're concerned with the solvability of Equation~\eqref{eqn: weil} 
\begin{equation*}
g^2 X^{\ell} + g Y^{\ell} + 1 = 0
\end{equation*}
over $\ff_{q}.$ Let $\Ccal_{g}$ be the affine plane curve defined by this equation. Note that $\Ccal_{g}$ is non-singular and irreducible over 
$\overline{\ff_{q}}$, an algebraic closure of $\ff_{q}.$  Denote  the set of $\ff_{q}$-rational points of $\Ccal_{g}$ by 
\[
\Ccal_{g}(\ff_{q}) = \left\{(x,y)\in \ff_{q}^{2}\mid g^2 x^{\ell} + g y^{\ell} + 1 = 0\right\}
\]
and let $N_{g} = |\Ccal_{g}(\ff_{q})|$ be its cardinality. Furthermore, let $\widetilde{\Ccal_{g}}$ be the (Zariski) closure of $\Ccal_{g}$ in the projective 
plane defined by the homogeneous equation 
\begin{equation}
\label{eqn: homogeneous-weil}
g^2 X^{\ell} + g Y^{\ell} + Z^{\ell} = 0. 
\end{equation}
Note that $\widetilde{\Ccal_{g}}$ is also non-singular. We let $\widetilde{N_{g}}$ denote the cardinality of $\widetilde{\Ccal_{g}}(\ff_{q}).$ 
Having Conjecture~\ref{conjB} and Conjecture~\ref{conjB'} in mind, we are especially concerned with whether or not a point $(x,y) \in \Ccal_{g}(\ff_{q})$  
satisfying  $xy \ne 0.$ The following lemma shows that this is always true except for very limited special cases. 

\begin{lem}
\label{lem2.1}
Equation~\eqref{eqn: homogeneous-weil} has a nontrivial solution $(x, y, z)$ with $x y z= 0$ 
if and only if one of the following situations holds\textup{:} 
\begin{enumerate}
\item[(i)] $\ell =1$ or $2$\textup{;} 
\item[(ii)] $\ell = 4$ and $-1 \not\in L$. 
\end{enumerate}
Moreover, if $\ell > 2$, then $x z \neq 0.$ 
\end{lem}

\begin{proof}
Suppose that $(x, y, z)$ is a nontrivial solution to Equation~\eqref{eqn: homogeneous-weil} with $x y z = 0$. 
Then only one  of $x, y, z$ is zero. Observe that if $x = 0$ or $z = 0$, then $-g \in L$ and $gL$ is either of order $1$ or $2$  in $\ff_q^{\times} / L$; 
if $y = 0$,  then $-g^2 \in L$ and $gL$ is of order $4$ in $\ff_q^{\times} / L$. 
In particular, we have $xz \neq 0$ provided that $\ell \neq 2$. In the case where $\ell = 4$, we see that $-L = g^2 L \neq L.$ It follows that 
$-1 \not\in L$. 

Conversely, if $\ell =1$ then it's clear that Equation~\eqref{eqn: homogeneous-weil} has a nontrivial solution $(x, y, z)$ with $x y z= 0.$ 
Suppose that $\ell = 2$, then  $\ff_q^{\times} / L = \{L, g L\}$. If $-1 \not\in L$, then $-L = g L$. In this case, $g = -a^{2} \in L$ for some
$a\in \ff_{q}^{\times}$. Then, we clearly have solutions $(x, y, z) = (1, a, 0)$ and $(0, 1, a). $ Suppose $-1 \in L$, then $-g^2 =b^{2}\in L$ 
for some $b\in \ff_{q}^{\times}$ and we have the solution $(x, y, z) = (1, 0, b)$ in this case. 

Finally, suppose that $\ell = 4$ and $-1 \not\in L$. Then both $g^2 L$ and $-L$ are of order $2$ in the cyclic group $\ff_q^{\times} / L$. Thus, $-L = g^2 L$ and this gives a solution $(x,y,z)=(1, 0, b)$ where $-g^2 = b^{4} \in L$. 
\end{proof}


Following~\cite{Weil1}, the number $N_{g}$ of solutions to Equation~\eqref{eqn: weil} can be expressed as a character sum which we now recall. 
As usual, by a multiplicative character of $\ff_{q}$ we mean a character of the group $\ff_q^{\times},$ i.e. a group homomorphism  from 
$\ff_{q}^{\times}$ to $\CC^{\times}.$ As we only deal with multiplicative characters of $\ff_{q}$, we'll simply call them characters. 
The trivial character will be denoted by $\e$ such that $\e(a) = 1$ for all $a\in \ff_{q}^{\times}.$  We extend  the domain of a character $\c$ 
such that $\chi(0) = 1$ if $\chi = \e$ and $\chi(0) = 0$ otherwise. We call the extension of $\c$ an extended character 
and still denote the extension by $\c$ if there is no danger of confusion.  We fix a character  $\c$ of order $\ell.$ Then we have 
\begin{align} 
\label{eqn:Ng} 
N_g = q + \sum_{1 \leq j, k \leq \ell-1} \chi^j(-g^{-2}) \chi^k (-g^{-1}) J(\chi^j, \chi^k) 
\end{align}
where 
$$
J(\chi^j, \chi^k) = \sum_{a \in \ff_q} \chi^j(a) \chi^k(1-a)
$$ 
is a {\it Jacobi sum} with respect to $\chi^j$ and $\chi^k$. The following properties of Jacobi sums are useful.
\begin{lem}[{\cite[Theorem~5.19, 5.21, 5.22]{LidNie}}]
\label{lem:jacobi sum}
Let $\lambda, \psi$ be two extended characters of $\ff_q$.
\begin{enumerate}
\item[(i)]
$J(\lambda, \psi) = J(\psi, \lambda)$\textup{;}
\item[(ii)]
$J(\e,\e) = q$\textup{;}
\item[(iii)]
$J(\lambda, \e) = 0$ if $\lambda \neq \e$\textup{;}
\item[(iv)]
$J(\lambda, \lambda^{-1}) = - \lambda(-1)$ if $\lambda \neq \e$\textup{;}
\item[(v)]
$|J(\lambda, \psi)| = \sqrt{q}$ if $\lambda, \psi$ and $\lambda \psi$ are all nontrivial.
\end{enumerate}
\end{lem}
Note that $|\chi^{i}(a)| = 1$ for all $a \in \ff_q^{\times}$. By (iv) and (v) of Lemma~\ref{lem:jacobi sum}, one has the following estimate of $N_{g}$ 
from~\eqref{eqn:Ng} 
$$
          |N_g - q| \leq M_0 + M_1 \sqrt{q}
$$ 
where $M_0$ (resp. $M_1$) is the number of pairs $(j, k)$ with $\chi^j \chi^k = \e$ (resp. $\chi^j \chi^k \neq \e$). Observe that 
$M_0 = \ell - 1$ and $M_1 = (\ell - 1) (\ell - 2)$. Thus, if 
\begin{align} 
\label{align:ineq}
q > (\ell-1) + (\ell-1) (\ell - 2) \sqrt{q}, 
\end{align} 
then $N_g > 0$. Consequently, for $q$ large enough (for example $q > (\ell-1)^4$), one has $N_g > 0$ for any $g\in \ff_q^{\times}$. 
%

For the  numbers of rational solutions to equations over finite fields, the  Hasse-Weil bound~\cite{Weil1} provides more precise information than
the crude estimate given above. 
\begin{thm1}[Hasse-Weil bound]
\label{thm:hasse-weil}
Let $\Ccal$ be a non-singular, absolutely irreducible projective curve over $\ff_{q}$ and let $N_{\Ccal} = |\Ccal(\ff_{q})|$ 
be the number of $\ff_{q}$-rational points of $\Ccal$. Then, 
\[
 | N_{\Ccal} - (q + 1)| \le 2 \gfk \sqrt{q}
\]
where $\gfk$ is the genus of $\Ccal.$ 
\end{thm1}
Applying the Hasse-Weil bound to $\widetilde{\Ccal_{g}}$, we see that 
\[
| \widetilde{N_{g}} - (q + 1)| \le (\ell-1)(\ell -2) \sqrt{q} 
\]
since the genus of $\widetilde{\Ccal_{g}}$ is  $\gfk_{\ell} = (\ell-1)(\ell -2)/2$ by the degree-genus formula~\cite{hartshore}. 
Consequently,  $\widetilde{N_{g}} > 0$ for any generator $g$ of $\ff_{q}^{\times}$ provided that 
$q + 1 >  (\ell-1)(\ell -2) \sqrt{q}$ and therefore  $\widetilde{\Ccal_{g}}(\ff_{q})$ is non-empty if $q \ge (\ell-1)^{2}(\ell -2)^{2}.$ 

With this lower bound, one can easily verify the truth of Conjecture~\ref{conjB'} (and Conjecture~\ref{conjB}) for small values of $\ell.$ The following results are direct consequences of Theorem~\ref{thm:hasse-weil}. For the reader's convenience, we give a proof. 

\begin{cor}
\label{cor: small ell}
For $1\le \ell \le 4$, we have $N_g > 0$ for every generator $g$ of $\ff_q^{\times}$. Moreover, there exists a point $(x, y)\in \Ccal_{g}(\ff_{q})$ such that 
$xy\ne 0$ and hence Conjecture~\ref{conjB'} holds for the case where $\ell_{0} = [\ff_{p}^{\times} : H] \le 4.$ 
\end{cor}

\begin{proof}
As it's easy to deduce the conclusion if $\ell = 1$, we leave the verification of this case to the reader. 
Let's first consider the case where $\ell= 2.$ Notice that in this case $p > 2$ and $\gfk_{\ell}=0.$ Therefore, $\widetilde{N_{g}} = q + 1$ by 
Theorem~\ref{thm:hasse-weil}. It's not hard to verify that 
\[
    N_{g} = \begin{cases}  \widetilde{N_{g}} & \text{if} \; q \equiv 1 \pmod{4} \\ \widetilde{N_{g}} -2  & \text{if}\; q \equiv 3 \pmod{4} \end{cases}. 
\]
Therefore, $N_{g} = q + 1$ if $q \equiv 1 \pmod{4}$ and $N_{g} = q -1$ if $q \equiv 3 \pmod{4}.$ In either case, we clearly have $N_{g} > 0.$ 
It remains to show that there exists a point $(x, y)\in \Ccal_{g}(\ff_{q})$ such that $xy \ne 0.$ Observe that there are at most four points in 
$\Ccal_{g}(\ff_{q})$ with either $x = 0$ or $y = 0.$  In the case where $q \equiv 1 \pmod{4}$ we have $N_{g} = q + 1 \ge 6 .$  It remains  to 
look at the case where $q \equiv 3 \pmod{4}.$ Since $\ell = 2 $ is a proper divisor of $q -1$ by assumption, we see that $q \ge 7$  and we also have 
$N_{g} = q - 1\ge 6.$ Now, it's clear that there's a point $(x, y)\in \Ccal_{g}(\ff_{q})$ such that $xy \ne 0$  since  $N_{g} > 4$ in both cases. 

Next, we consider the cases where $\ell = 3$ and $4.$ Since $\ell > 2$,  we have that $N_{g} = \widetilde{N_{g}}$ by Lemma~\ref{lem2.1}. 
Suppose that $\ell = 3.$ In this case $\widetilde{\Ccal_{g}}$ is of genus one.  Then the Hasse-Weil bound gives that 
\[
  N_{g} \ge (q+1) - 2\sqrt{q} = (\sqrt{q} - 1)^{2} > 0. 
\]
Therefore, $N_{g} > 0$ for any generator $g$ of $\ff_q^{\times}$ in this case. Suppose that there exists a solution $(x,y)$ to Equation~\eqref{eqn: weil} 
such that either $x=0$ or $y=0$ for $\ell = 3.$ Then we get that either $g$ or $g^{2}$ is a cube in $\ff_{q}^{\times}.$ This implies that the order of $gL$
in the group $ \ff_{q}^{\times}/L$ divides $2$ which is absurd since $|\ff_{q}^{\times}/L| = 3.$ Therefore, any $(x,y)\in \Ccal_{g}(\ff_{q})$ must satisfy 
$xy\ne 0$ as desired. 

Assume that $\ell = 4$. A direct computation shows that for $q > 49$, we have that $N_{g} > 8.$  Since there are at most eight solutions to Equation~\eqref{eqn: weil} such that either $x=0$ or $y=0$ for $\ell = 4$, we see that for $q > 49$ there exists  
$(x,y)\in \Ccal_{g}(\ff_{q})$ such that $xy \ne 0$ as asserted. It remains to check prime power numbers $q$ 
satisfying $q \le 49$ such that  $4$ is a proper divisor of $q-1.$  
Hence, we are left with eight cases where $q = 9, 13, 17, 25, 29, 37, 41, 49$ to verify. Note that if $(x,y)\in \Ccal_{g}(\ff_{q})$ with $xy\ne 0$ then
$(x^{-1}, x^{-1}y)\in \Ccal_{g^{-1}}(\ff_{q})$. It follows that $\Ccal_{g}(\ff_{q})$ contains a point whose coordinates are nonzero if and only if
$\Ccal_{g^{-1}}(\ff_{q})$ has this property as well.  Also, for any generator $g'$ of $\ff_{q}^{\times}$ we have   $g' L \in \{g L, g^{-1} L\}$ in 
the case where $|\ff_{q}^{\times}/L| = 4.$  It follows that $\Ccal_{g}(\ff_{q})$ contains a point whose coordinates are nonzero if and only if 
$\Ccal_{g'}(\ff_{q})$ has this property as well. Hence, it suffices to show  that $\Ccal_{g}(\ff_{q})$ containing a point with nonzero coordinates for 
just one generator $g$ of  $\ff_{q}^{\times}.$  
We give the following table for each case.
$$
\begin{array}{|c|c|c|c|c|c|c|c|c|} \hline 
q & 9 & 13 & 17 & 25 & 29 & 37 & 41 & 49 \\ \hline 
g & \alpha & 2 & 3 & \beta & 2 & 5 & 6 & \gamma \\ \hline 
x & \alpha & 4 & 6 & 1 & 4 & 2 & 3 & 2 \\ \hline 
y & \alpha & 1 & 2 & \beta^2 & 4 & 2 & 3 & 2\gamma^{7} \\ \hline
\end{array}
$$ 
where $\alpha = 1 + \sqrt{-1}$ in $\ff_9 = \ff_3(\sqrt{-1})$, $\beta = 1 + 2 \sqrt{2}$ in $\ff_{25} = \ff_5(\sqrt{2})$ and $\gamma = 4 + \sqrt{-1}$ in $\ff_{49} = \ff_{7}(\sqrt{-1})$.

This completes the verifications of all cases in which $N_{g} > 0.$ Moreover,  we've exhibited all solutions $(x,y)$ such that $xy \ne 0$ and thus finish 
the proof. 
\end{proof}

\noindent We would like to point out that it's possible to prove Corollary~\ref{cor: small ell} by using the  bound~\eqref{align:ineq} 
without applying the Hasse-Weil bound. 

In view of Conjecture~\ref{conjB'}, we only need to find a generator $g$ of $\ff_{q}^{\times}$ such that $\Ccal_{g}(\ff_{q})$ is non-empty. Instead of 
computing $N_{g}$, our goal is to show that the following sum  
\[
N(q, \ell)  = \sum_{ \ff_q^{\times}=\langle g'\rangle} N_{g'}  = \sum_{\substack{{1 \leq t \leq q-1}, \\ \gcd (t, q-1) = 1}} 
N_{g^t} 
\]
is a positive integer under appropriate conditions.

\section{Key Ingredients}
\label{sec:sec3}

In this section, we gather tools and results that are needed for the proof of Theorem~\ref{thm1.1}. 
To simplify the notation, we'll put $(a_{1}, a_{2}) = \gcd(a_{1}, a_{2}),$ the greatest common divisor of integers $a_{1}$ and $a_{2}.$ 
The following lemma is an elementary fact in algebra which we will use repeatedly. As one can easily find a proof in any algebra text book, we skip the proof here. 

\begin{lem}
\label{lem:elementary lemma}
Let $n \in \nn$ and let $d$ be a divisor of $n$. Then the canonical group homomorphism 
$$
\pi:  (\zz / n \zz)^{\times} \to (\zz / d \zz)^{\times}
$$ 
induced by 
$$
\begin{array}{ccc} 
\zz / n \zz & \to & \zz / d \zz \\ k + n \zz & \mapsto & k + d \zz \end{array}
$$ 
is surjective. Furthermore, this homomorphism splits. Namely, there exists a subgroup $M$ of $(\zz / n \zz)^{\times}$ which is isomorphic to  
$(\zz / d \zz)^{\times}$ under $\pi$ and $(\zz / n \zz)^{\times} = M \cdot N$ where $N = \ker(\pi).$ 
\end{lem}

%

For $n \in \nn$ and $m \in \zz$, the {\it Ramanujan's sum} $c_n(m)$ (\cite{Ram} or \cite[pp. 179--199]{HarSesWil}) is defined by 
$$
c_n(m) = \sum_{\substack{1 \leq t < n, \\ (t, n) = 1}} \zeta_n^{mt}
$$ 
where $\zeta_n$ is a primitive $n$-th root of $1$ in $\cc$. Studying on cyclotomic polynomials, O. H{\"{o}}lder \cite{Hol} showed that the sum $c_n(m)$ has a nice closed form in terms of the Euler and M\"{o}bius functions. Denote $\varphi$ the Euler's totient function and $\mu$ the M\"{o}bius function. We present it in the following lemma where the right-hand side is also called {\it von Sterneck function} \cite{Ste}. A proof is given below to the readers for convenience. For different proofs, one refers to \cite{AndApo}, \cite{Mol} and \cite[Theorem~272]{HarWri}.

\begin{lem}
\label{lem:ramanujan sum}
Let $n \in \nn$ and $m \in \zz$. Then 
$$
c_n(m) = \mu\left(\frac{n}{(n, m)}\right) \frac{\varphi(n)}{\varphi\left({\frac{n}{(n, m)}}\right)}.
$$
\end{lem}

\begin{proof}
First of all, suppose that $m = 1$. Recall an elementary formula that 
$$
\sum_{k \mid r} \mu(k) = \left\{ \begin{array}{ll} 1 & \text{if } r = 1; \\ 0 & \text{if } r > 0, \end{array} \right.
$$ 
for $r \in \nn$ (\cite[Theorem~6.6]{Bur}). Then 
$$
c_n(1) = \sum_{t = 1}^n \zeta_n^t \sum_{k \mid (t, n)} \mu(k) = \sum_{k \mid n} \mu(k) \sum_{\substack{1 \leq t \leq n, \\ k \mid (t, n)}} \zeta_n^t 
= \sum_{k \mid n} \mu(k) \sum_{1 \leq t' \leq \frac{n}{k}} (\zeta_n^k)^{t'}
$$ 
where $t' = \frac{t}{k}$. Since $\zeta_n^k$ is a primitive $\frac{n}{k}$-th root of $1$, the last sum gives $1$ if $k = n$ and $0$ if $k < n$. Hence, $c_n(1) = \mu(n)$.

For general $m \in \zz$, we rewrite $c_n(m)$ as 
$$
c_n(m) = \sum_{t \in (\zz / n \zz)^{\times}} z^t
$$ 
where $z = \zeta_n^m$. Let $d = \frac{n}{(n,m)}$.  Recall that $(\zz / n \zz)^{\times} = M \cdot N$ where $M$ and $N$ are given in Lemma~\ref{lem:elementary lemma}. 
Note that $z = \zeta_n^m$ is a primitive $d$-th root of $1$ since $\left(d, \frac{m}{(n,m)}\right) = 1$. Hence, 
$$
c_n(m) = \sum_{t_1 \in M} \; \sum_{t_2 \in N} z^{t_1 t_2} = \sum_{t_1 \in M} |N| z^{t_1} = |N| c_d(1).
$$ 
Now, the result follows from $|N| = \frac{\varphi(n)}{\varphi(d)}$ and $c_d(1) = \mu(d)$ by the first paragraph.
\end{proof}

Note that $(n, m) = (n, (n, m))$. Thus, we have the following immediate consequence.

\begin{cor}
\label{cor:ramanujan sum}
Let $n \in \nn$ and $m \in \zz$. Then $c_n(m) = c_n((n, m))$. 
\end{cor}

\noindent As a consequence, we note that if $m'$ is an integer such that $m' \equiv m \pmod{n}$, then we conclude from 
Corollary~\ref{cor:ramanujan sum} that  $c_n(m') = c_n(m)$.

The following decomposition of a Cartesian product is useful for counting pairs of integers. Roughly speaking, the product below is partitioned by parallel lines on $\zz \times \zz$.

\begin{lem}
\label{lem:decomposition}
Let $n \in \nn$ and let $I = \{1, 2, \ldots, n-1\}$. For every $a, b \in \zz$, we have 
$$
I \times I = \biguplus_{d \mid n} \; \biguplus_{\substack{1 \leq t \leq \frac{n}{d}, \\ (t, \frac{n}{d}) = 1}} \{(x, y) \in I \times I \mid a x + b y \equiv t d \pmod{n}\}.
$$
\end{lem}

\begin{proof}
Let $a, b \in \zz$ be two fixed integers and let 
$$
S(d, t) = \{(x, y) \in I \times I \mid a x + b y \equiv t d \pmod{n}\}.
$$ 
Then the union of $S(d, t)$ for all such $d$ and $t$ is a subset of $I \times I$. This union indeed contains all elements of $I \times I$. To see this, let 
$(x, y) \in I \times I$ and let $d = (a x + b y, n)$. Pick $1 \leq t \leq \frac{n}{d}$ such that $t \equiv \frac{a x + b y}{d} \pmod{\frac{n}{d}}$. 
Then $(t, \frac{n}{d}) = 1$ and $a x + b y \equiv t d \pmod{n}$. Thus, $(x, y) \in S(d, t)$. Finally, if $S(d, t) \cap S(d', t') \neq \varnothing$, 
then choose one pair $(x, y)$ in this intersection. We obtain $d = (a x + b y, n) = d'$. Furthermore, $t d \equiv t' d \pmod{n}$ implies that $t \equiv t' \pmod{\frac{n}{d}}$. Since $1 \leq t, t' \leq \frac{n}{d}$, we have $t = t'$.
\end{proof}

\section{Proof of the Main Result}
\label{sec:proof}

Recall that we aim at showing the following sum 
\[
N(q, \ell) = \sum_{\substack{{1 \leq t \leq q-1}, \\ (t, q-1) = 1}} N_{g^t}
\]
is not equal to zero where $g$ is a fixed generator of $\ff_{q}^{\times}.$  It's not hard to see that if $g^{t} L = g^{s} L $ then 
$N_{g^{t}} = N_{g^{s}}$. We have the following reduction for $N(q, \ell).$


\begin{prop}
\label{prop:reduction}
Let $g$ be generator of $\ff_q^{\times}$ and $\ell \mid q-1$. 
Then 
$$
N(q, \ell) = \frac{\varphi(q-1)}{\varphi(\ell)} \sum_{\substack{{1 \leq t \leq \ell}, \\ (t, \ell) = 1}} N_{g^t}.
$$
\end{prop}

\begin{proof}
By the definition of $N(q,\ell)$, it is a sum indexed by all elements of $(\zz / (q-1) \zz)^{\times}$.
Recall that $(\zz / (q-1) \zz)^{\times} = M \cdot N$ where $M$ and $N$ are given in Lemma~\ref{lem:elementary lemma} for $n = q-1$ and $d = \ell$. 
Note that if $t_1 \in N$, then $t_{1}\equiv 1 \pmod{\ell}$ and $g^{t_1} L = g L.$ 
It follows that $N_{g^{t_1 t_2}} = N_{g^{t_2}}$ for any integer $t_2$. Hence, $$N(q, \ell) = \sum_{t \in (\zz / (q-1) \zz)^{\times}} N_{g^t} = \sum_{t_1 \in N} \; \sum_{t_2 \in M} N_{g^{t_1 t_2}} = |N| \sum_{t_2 \in M} N_{g^{t_2}}.$$ The result follows since $|N| = \frac{\varphi(q-1)}{\varphi(\ell)}$ and $M \simeq (\zz / \ell \zz)^{\times}$.
\end{proof}

Recall from~\eqref{eqn:Ng} that $$N_{g^t} = q + \sum_{1 \leq j, k \leq \ell-1} \chi^j(-g^{-2t}) \chi^k(-g^{-t}) J(\chi^j, \chi^k)$$ where $\chi$ is a nontrivial character of $\ff_q$ of order $\ell$. We can rewrite $N_{g^t}$ as follows.

\begin{lem}
\label{lem:4.1}
For $1 \leq t \leq \ell$ and $(t, \ell) = 1$, 
$$
N_{g^t} = q + 1 + \sum_{\substack{1 \leq j, k \leq \ell-1, \\ j + k \neq \ell}} \chi(-1)^{j+k} \chi(g^{-1})^{(2 j + k) t} J(\chi^j, \chi^k).
$$ 
\end{lem}

\begin{proof}
If $j + k = \ell$, then $J(\chi^j, \chi^k) = -\chi(-1)^j$ by (iv) of Lemma~\ref{lem:jacobi sum}. It follows that 
$$
\sum_{\substack{1 \leq j, k \leq \ell-1, \\ j + k = \ell}} \chi(-1)^{j+k} \chi(g^{-1})^{(2 j + k) t} J(\chi^j, \chi^k) = -\sum_{j=1}^{\ell-1} \chi(-g^{-t})^{j}. 
$$ 
Note that the kernel of $\chi$ is $L$. If $-g^{-t} \in L$, then $g^{-t} L = -L$. This gives $\ell = |\ff_q^{\times} / L| \leq 2$ since $g^{-t} L$ also generates $\ff_q^{\times} / L$ as $1 \leq t \leq \ell$ and $(t, \ell) = 1$. This is not our case and thus $\c(-g^{-t}) \ne 1.$ Hence, 
$\sum_{j=1}^{\ell-1} \chi(-g^{-t})^{j} = -1$ and the result follows. 
\end{proof}

Now, we are ready to prove our main theorem. 

\begin{thm}
\label{thm4.1}
Let $q$ be a power of a prime and let $\ell$ be a proper divisor of $q-1$. If 
$$
q \geq (2^{\om(\ell)} (\ell-3 - \delta) + 2)^2 - 2
$$ 
where $\om(\ell)$ is the number of distinct prime divisors of $\ell$ and 
$$
\delta = \left\{\begin{array}{ll} 1 & \text{if } 4 \mid \ell, \\ 0 & \text{otherwise,}\end{array} \right.
$$ 
then there is a generator $g$ of $\ff_q^{\times}$ such that $N_{g} > 0.$ 
\end{thm}

\begin{proof}
By Proposition~\ref{prop:reduction}, it is enough to consider the subsum 
$$
N = \sum_{\substack{{1 \leq t \leq \ell}, \\ (t, \ell) = 1}} N_{g^t}
$$ 
where $g$ is a fixed generator of $\ff_{q}^{\times}.$  Lemma~\ref{lem:4.1} gives that 
\begin{align*}
N  & = \varphi(\ell) (q+1) + \sum_{\substack{1 \leq j, k \leq \ell-1, \\ j + k \neq \ell}} \chi(-1)^{j+k} J(\chi^j, \chi^k) z(j, k)\quad \text{where}  \\
& z(j, k)  = \sum_{\substack{{1 \leq t \leq \ell}, \\ (t, \ell) = 1}} \chi(g^{-1})^{(2 j + k) t}.
\end{align*}
Note that $\chi(g^{-1}) = \zeta_{\ell}$ is a primitive $\ell$-th root of $1$. Therefore, 
$$
z(j, k) = \sum_{\substack{{1 \leq t \leq \ell}, \\ (t, \ell) = 1}} \zeta_{\ell}^{(2j + k) t} = c_{\ell}(2 j + k),
$$ 
is a Ramanujan's sum and 
$$ 
N = \varphi(\ell) (q+1) + \sum_{\substack{1 \leq j, k \leq \ell-1, \\ j + k \neq \ell}} \chi(-1)^{j+k} J(\chi^j, \chi^k)  c_{\ell}(2 j + k). 
$$

Let $I = \{1, 2, \ldots, \ell-1\}$ and for positive integer $d \mid \ell$ and integer $t$ with $1\le t \le \ell/d$ such that $(t, d/\ell) = 1,$ 
we set 
$$
S'(d, t) = \{(j, k) \in I \times I \mid 2 j + k \equiv t d \pmod{\ell} \text{ and } j + k \neq \ell\}.
$$ 
By omitting pairs $(j, k)$ of $I \times I$ satisfying $j + k = \ell$, Lemma~\ref{lem:decomposition} gives that 
$$
N - \varphi(\ell) (q+1) = \sum_{d \mid \ell} \; \sum_{\substack{1 \leq t \leq \frac{\ell}{d}, \\ (t, \frac{\ell}{d}) = 1}} \; \sum_{(j, k) \in S'(d, t)} \chi(-1)^{j+k} J(\chi^j, \chi^k)  c_{\ell}(2 j + k).
$$ 
For $(j, k) \in S'(d, t)$, one has $(2 j + k, \ell) = d$ and then $c_{\ell}(2 j + k) = c_{\ell}(d)$ by Corollary~\ref{cor:ramanujan sum}. Thus, 
$$
N - \varphi(\ell) (q+1) = \sum_{d \mid \ell} c_{\ell}(d) f(d)
$$ 
where 
$$
f(d) = \sum_{\substack{1 \leq t \leq \frac{\ell}{d}, \\ (t, \frac{\ell}{d}) = 1}} \; \sum_{(j, k) \in S'(d, t)} \chi(-1)^{j+k} J(\chi^j, \chi^k).
$$ 
We need to estimate $|f(d)|$. 

By definition, every pair $(j, k)$ of $S'(d, t)$ satisfies $j + k \not\equiv 0 \pmod{\ell}$ and thus $|J(\chi^j, \chi^k)| = \sqrt{q}$ by (v) of Lemma~\ref{lem:jacobi sum}.  Since $|\chi(-1)| = 1$, it follows that 
$$
|f(d)| \leq \sum_{\substack{1 \leq t \leq \frac{\ell}{d}, \\ (t, \frac{\ell}{d}) = 1}} |S'(d, t)| \sqrt{q}.
$$ 

Now, we compute $|S'(d, t)|$. Observe that every pair $(j, k)$ in $S'(d, t)$ is determined by $j \in I$ with the proviso that $j+ k \ne \ell.$ 
Thus, for $(j, k)\in I\times I$  satisfying the congruence $2j + k \equiv t d   \pmod{\ell}$  
we have to exclude the pair $(j, k)$ with $j \equiv t d \pmod{\ell}.$ 
Note that  $ t d \le \ell$ while $j \le \ell-1$, 
this congruence can  occur only when $d \lneqq \ell$ and $j = t d$. 
Moreover, as $k \neq 0$, we also need to exclude the case where $2 j \equiv t d \pmod{\ell}$. This depends on the parity of $\ell$. 
We discuss in the next paragraph to steer clear of confusing. 

Suppose that  $\ell$ is odd. Let $s\in I$ be such that $2 s \equiv 1 \pmod{\ell}$. Then, we need to exclude  $j \in I$ 
such that $j \equiv s t d \pmod{\ell}$. If  $d = \ell$, then there is no such $j$ because $j \not\equiv 0\pmod{\ell}$. 
When $d \lneqq \ell$, there is exactly one $j_0 \in I$ satisfying $j_0 \equiv s t d \pmod{\ell}$. Remember that we also have to exclude the case where 
$j = t d$. As a consequence, if $\ell$ is odd, then 
$$
|S'(d, t)| = \left\{ \begin{array}{ll} |I| & \text{if } d = \ell; \\ |I| - 2 & \text{if } d \neq \ell. \end{array} \right.
$$ 

Now we assume that $\ell$ is even. There are three cases to consider: (i) $t d$ is odd, (ii) $t d$ is even and $d\lneqq \ell$ and (iii) $t = 1, d = \ell$. 
For case (i),  since $t d$ is odd,  there is no $j$ such that $2 j \equiv t d \pmod{\ell}$. Only the case where $j = t d$ has to be excluded. For (ii) and (iii), 
we have that $t d$ is even and then there is  some $j_1 \in I$ such that $2 j_1 \equiv t d \pmod{\ell}$. In fact, we have $j_1 \equiv \frac{t d}{2} \pmod{\frac{\ell}{2}}$. If $d\lneqq \ell$, then either $j_1 = \frac{t d}{2}$ or $j_1 = \frac{t d}{2} + \frac{\ell}{2}$ and in particular, $j_1 \neq t d$ in this case. If $d = \ell$, then $t = 1$ and $j_1 = \frac{\ell}{2}$. 
We conclude that 
\[
|S'(d, t)| = \left\{ \begin{array}{ll} 
                    |I| - 1 & \text{if } d = \ell; \\ 
                    |I| - 1 & \text{if $d \neq \ell$ and $t d$ is odd}; \\ 
                    |I| - 3 & \text{otherwise.} 
                    \end{array} \right.
\]

Combining these two situations of $\ell$, we have an expression for $|S'(d,t)|$  as follows. 
$$
|S'(d, t)| = |I| - 2 + 2 \left\lfloor \frac{d}{\ell} \right\rfloor - (-1)^{td} \lambda = |I| - 2 + 2 \left\lfloor \frac{d}{\ell} \right\rfloor - (-1)^{d} \lambda 
$$ 
where $\lambda = \frac{1}{2} (1 + (-1)^{\ell})$ and note that $t d \equiv d \pmod{2}$ when $\ell$ is even as $(t, \frac{\ell}{d}) = 1$. 
In particular, $|S'(d, t)|$ is independent on $t$.
Therefore, 
$$
|f(d)| \leq \varphi\left( \frac{\ell}{d} \right) \left(|I| - 2 + 2 \left\lfloor \frac{d}{\ell} \right\rfloor -  (-1)^d \lambda\right) \sqrt{q}
$$ 
and 
\begin{align*}
|N - \varphi(\ell) (q+1)| & \leq \sum_{d \mid \ell} |c_{\ell}(d) f(d)| \\
& \leq \sum_{d \mid \ell} |c_{\ell}(d)| \varphi\left( \frac{\ell}{d} \right) \left(|I| - 2 + 2 \left\lfloor \frac{d}{\ell} \right\rfloor - (-1)^d \lambda \right) \sqrt{q}.
\end{align*}
By Lemma~\ref{lem:ramanujan sum}, we have $c_{\ell}(d) \varphi(\frac{\ell}{d}) = \mu(\frac{\ell}{d}) \varphi(\ell)$. By dividing $\varphi(\ell)$, we obtain 
$$
\left|\frac{N}{\varphi(\ell)} - (q+1)\right| \leq \sum_{d \mid \ell} \left| \mu\left( \frac{\ell}{d} \right) \right| \left(|I| - 2 + 2 \left\lfloor \frac{d}{\ell} \right\rfloor - (-1)^d \lambda \right) \sqrt{q}.
$$ 
Note that $|\mu(\frac{\ell}{d})| = 1$ if $\frac{\ell}{d}$ is square-free and $|\mu(\frac{\ell}{d})| = 0$ otherwise. 
Observe that the number of divisors $d$ in which $\frac{\ell}{d}$ is square-free is $2^{\om(\ell)}$. Moreover, 
$$
\sum_{\substack{d \mid \ell, \\ \mu\left(\frac{\ell}{d}\right) \neq 0}} (-1)^d \lambda  = 2^{\om(\ell)} \delta.
$$ 
It follows that 
$$
\left|\frac{N}{\varphi(\ell)} - (q+1)\right| \leq \left( 2^{\om(\ell)} (|I| - 2 - \delta) + 2 \right) \sqrt{q}.
$$ 

As a consequence, if $q + 1 > ( 2^{\om(\ell)} (|I| - 2 - \delta) + 2) \sqrt{q}$, then we can conclude that $N > 0$. By dividing $\sqrt{q}$ and taking square on both sides, the last inequality is equivalent to 
$$
q + 2 + \frac{1}{q} > ( 2^{\om(\ell)} (|I| - 2 - \delta) + 2)^2 = ( 2^{\om(\ell)} (\ell - 3 - \delta) + 2)^2.
$$ 
All terms except $\frac{1}{q}$ on both sides are integers and $\frac{1}{q} < 1$. The result follows.
\end{proof}

\begin{cor}
\label{cor4.1}
If $q$ is congruent to $1$ modulo $5$, then $g^2 X^{5} + g Y^{5} + 1 = 0$ is solvable in $\ff_q$ for some primitive root $g$ of $\ff_q$.
\end{cor}

\begin{proof}
By Theorem~\ref{thm4.1}, the result holds for $q \geq 34$. For $q < 34$ and $q \equiv 1 \pmod{5}$, one has $q = 11, 16, 31$. For $q = 11$, one has $g = 7$ and $7^2 \cdot (-1)^5 + 7 \cdot (-1)^5 + 1 = 0$. For $q = 16$, one has $\ff_{16} = \ff_2(\alpha)$ with $\alpha^4 = 1 + \alpha$ and $\ff_{16}^{\times} = \langle \alpha \rangle$. Moreover, $\alpha^2 \cdot (\alpha^2)^5 + \alpha \cdot (\alpha^2)^5 + 1 = 0$. For $q = 31$, one has $g = 3$ and $3^2 \cdot (-3)^5 + 3 \cdot 3^5 + 1 = 0$. 
\end{proof}

\begin{cor}
\label{cor4.2}
If $q$ is congruent to $1$ modulo $6$, then $g^2 X^{6} + g Y^{6} + 1 = 0$ is solvable in $\ff_q$ for some primitive root $g$ of $\ff_q$ if and only if $q > 13$.
\end{cor}

\begin{proof}
By Theorem~\ref{thm4.1}, the result holds for $q \geq 194$. For $q < 194$ and $q \equiv 1 \pmod{6}$, we first look at $q = 7$ and $q = 13$. 
When $q = 7$, one has $x^6 = 1$ for all $x \in \ff_7^{\times}$. But $g^2 + g + 1 \neq 0$ for any primitive root $g$ of $\ff_7$. So it is not solvable in this case. For $q = 13$, $x^6 = \pm 1$ for all $x \in \ff_{13}^{\times}$. It is easy to check that $g^2 X^6 + g Y^6 + 1 = 0$ is not solvable for all primitive roots $g = 2, 6, 7, 11$ of $\ff_{13}$. For the rest of cases that $13 < q < 194$, the following table gives a solution for some primitive root $g$: 
$$
\begin{array}{|c|c|c|c|c|c|c|c|c|c|c|c|c|} \hline 
q & 19 & 25 & 31 & 37 & 43 & 49 & 61 & 67 & 73 & 79 & 97& 103 \\ \hline 
g & 2 & \alpha & 3 & 2 & 3 & \beta & 2 & 2 & 5 & 3 & 5 & 5 \\ \hline 
X & 1 & \alpha^3 & 19 & 2 & 1 & \beta^3 & 24 & 4 & 1 & 6 & 5 & 5 \\ \hline 
Y & 2 & \alpha & 27 & 1 & 28 & \beta^3 & 4 & 43 & 59 & 6 & 29 & 32 \\ \hline
\end{array}
$$ 

$$
\begin{array}{|c|c|c|c|c|c|c|c|c|c|c|} \hline 
q & 109 & 121 & 127 & 139 & 151 & 157 & 163 & 169 & 181 & 193 \\ \hline 
g & 6 & \gamma & 3 & 2 & 6 & 5 & 2 & \kappa & 2 & 5 \\ \hline 
X & 16 & \gamma^7 & 84 & 2 & 1 & 22 & 8 & 1 & 86 & 1 \\ \hline 
Y & 26 & \gamma^4 & 3 & 103 & 132 & 82 & 1 & \kappa^2 & 148 & 127 \\ \hline\end{array}
$$
where $\alpha = 3 + \sqrt{2}$ in $\ff_{25} = \ff_5(\sqrt{2})$, $\beta = 4 + \sqrt{-1}$ in $\ff_{49} = \ff_7(\sqrt{-1})$, $\gamma = 2 + \sqrt{2}$ in $\ff_{121} = \ff_{11}(\sqrt{2})$ and $\kappa = 7 + 2 \sqrt{2}$ in $\ff_{169} = \ff_{13}(\sqrt{2})$.
\end{proof}

\section{Application: Conflict-Avoiding Codes of Weight $3$}
\label{sec5}

In this section, we apply our main result to the construction of CAC. We consider the special case where $q=p$ and $\ell=\ell_{0}$ the index of  
$H$ in $\ff_{p}^{\times}.$ Let's start with a proof of Corollary~\ref{thm1.2}. 

\begin{proof}
Let $p$ be a prime such that the multiplicative order $o_{p}(2)$ of $2$ is not a multiple of $4.$ Suppose that $p$ satisfies the condition
\[
p \geq (2^{\om(\ell_{0})} (\ell_{0} - 3 - \delta) + 2)^2 - 2
\]
where $\d = 1$ if $4 \mid \ell_{0}$ and   $\d = 0$ otherwise. It follows from Theorem~\ref{thm1.1} that there exists a generator 
$g$ of $\ff_{p}^{\times}$ such that Equation~\eqref{eqn: CAC} is solvable over  $\ff_{p}.$ 
By Corollary~\ref{cor: small ell} for  $1\le \ell_{0}\le 4$ and Lemma~\ref{lem2.1} for $\ell_{0} \ge 5$, 
we see  that there exists a solution $(x, y) \in \ff_{p}^{2}$ to Equation~\eqref{eqn: CAC} satisfying $xy \ne 0.$ 
Thus,  Conjecture~\ref{conjB'} and hence Conjecture~\ref{conjB} holds for  prime numbers $p$ 
satisfying the inequality given above.  As it is explained in Section~\ref{sec1}, Conjecture~\ref{conjA} is also true for these prime numbers. 
Combining the algorithm given in~\cite{FuLoShu}, we conclude that an optimal CAC of length $p$ and weight 3 has the size
$$
M(p, \ell_0) = \frac{p-1-2\ell_0}{4} + \left\lfloor \frac{\ell_0}{3} \right\rfloor
$$ 
as desired. 
\end{proof}

Our strategy for studying the size of optimal CAC of prime lengths is through investigating  Conjecture~\ref{conjB} (equivalently, Conjecture~\ref{conjB'}). 
In the paper~\cite{MaZhaShe} the authors announce that Conjecture~\ref{conjB} has been verified to hold for prime  $p \leq 2^{30}$. 
Applying Corollary~\ref{thm1.2}, we are able to extend the range of prime numbers $p$ such that Conjecture~\ref{conjB} holds. 

For $\ell \in \nn$, let 
$$
b(\ell) = (2^{\om(\ell)} (\ell-3 - \delta)+2)^2 - 2
$$ 
be the lower bound appearing in Theorem~\ref{thm1.1}.  
Suppose that  $p$ is a prime number such that the index $\ell_{0}$ of $H$ in $\ff_{p}^{\times}$ satisfying $b(\ell_{0})\le 2^{30}$, 
then either $p \leq 2^{30}$ or $p > 2^{30} \geq b(\ell_0)$. It follows that Conjecture~\ref{conjB'} holds for this prime number $p$. 
This leads to the question about the integer $\ell$ such that $b(\ell)\le 2^{30}.$ The answer depends on the number of distinct prime divisors 
of $\ell.$ 
For $\om(\ell) < 4$, one can check that $b(\ell) \leq 2^{30}$ if one of the following conditions holds: 
$$
\left\{ \begin{array}{l} 
\ell < 16411 \text{ with } \om(\ell) = 1, \\ 
\ell < 8197 \text{ with } \om(\ell) = 2, \\ 
\ell < 4100 \text{ with } \om(\ell) = 3. 
\end{array} \right.
$$
Moreover, one also has $b(\ell) \leq 2^{30}$ whenever $\ell < 2070$. As a consequence, we have the following result.


\begin{thm}
\label{thm:C}
\textup{Conjecture~\ref{conjB'}} holds for primes $p$ with $\ell_0$ satisfying one of the following conditions 
$$
\left\{ \begin{array}{l} 
\ell_0 < 16411 \text{ with } \om(\ell_0) = 1, \\ 
\ell_0 < 8197 \text{ with } \om(\ell_0) = 2, \\ 
\ell_0 < 4100 \text{ with } \om(\ell_0) = 3, \\ 
\ell_0 < 2070.
\end{array} \right.
$$ 
\end{thm}

For an integer $\ell$ such that $b(\ell) > 2^{30}$, we consider the set of prime numbers between $2^{30}$ and $b(\ell).$  
Let 
$$
P(\ell) = \left\{\text{primes } p > 2^{30} \mid [\ff_p^{\times} : \langle -1, 2 \rangle] = \ell \text{ and } p < b(\ell) \right\}.
$$ 
For prime numbers in $P(\ell)$, we  verify~\textup{Conjecture~\ref{conjB'}} by the aid of computer for the computations. For instance, 
we obtain that $P(2070) = \varnothing$ and hence~\textup{Conjecture~\ref{conjB'}} holds for prime numbers $p$ with $\ell_{0}= 2070.$ By computer
search, there are $423$ primes in the union of $P(\ell)$ for $2070 \leq \ell \leq 3000$. The largest prime number in the union  is $7324065841$ 
with $\ell=  2730.$ We have checked that Conjecture~\ref{conjB'} holds for these prime numbers.
 

\begin{thm}
\label{thm:D}
\textup{Conjecture~\ref{conjB'}} holds for primes $p$ such that $[\ff_p^{\times} : \langle -1, 2 \rangle] \leq 3000$.
\end{thm}

\noindent For the solvability of Equation~\eqref{eqn: weil}, we have the following simple observation. 

\begin{prop}
\label{prop:observation}
Let $\ell$ be a proper divisor of $q-1$ and $\ell' \mid \ell.$ Suppose that Equation~\eqref{eqn: weil} is solvable over $\ff_{q}$ for exponent $\ell$ then 
it is also solvable for exponent $\ell'$. 
\end{prop}

Proposition~\ref{prop:observation} leads to the following consideration for $q = p$, a prime number, and $\ell = \frac{p-1}{2}$ in Equation~\eqref{eqn: weil}. In this case, $L = \{\pm 1\}$ and it suffices to consider the four possibilities $\pm g^2 \pm g + 1 = 0$ where $g\in \nn$ is a primitive root modulo $p$ 
(i.e. $g$ is a generator of $\ff_p^{\times}$). Clearly, $g^2 \pm g + 1 = 0$ if and only if $g^3 = \pm 1$, and then $p-1 = 3, 6$. 
This holds for $p = 7$ and $g = 3$. The remaining two cases $- g^2 \pm g + 1 = 0$ are equivalent because $- g^2 + g + 1 = 0$ if and only if 
$-g^{-2} - g^{-1} + 1 = 0$. In other words, we only need to consider the equality $g^2 = g + 1$ over $\ff_p$. Such a primitive root $g$ is called a 
{\it Fibonacci primitive root} modulo $p$, as the {\it golden ratio} $\varphi_{gr}$ satisfying $\varphi_{gr}^2 = \varphi_{gr} + 1$. 
The primes $p$ such that $\ff_p$ has a Fibonacci primitive root is the sequence 
\href{https://oeis.org/A003147}{A003147}: $5, 11, 19, 31, 41, 59, 61, 71, 79, 109, \ldots$ on \href{https://oeis.org/}{OEIS} \cite{OEIS}. Consequently, 
we have $N(p, (p-1)/2) > 0$ if and only if  $\ff_p$ has a Fibonacci primitive root. On the other hand, the order $|H|$ of the subgroup $H = \langle -1, 2 \rangle$ is an even integer. It follows that $\ell_{0} = \frac{p-1}{|H|}$ must divide $\frac{p-1}{2}$. By Proposition~\ref{prop:observation}, we see that 
Equation~\eqref{eqn: CAC} is solvable over $\ff_{p}$ if  $\ff_p$ has a Fibonacci primitive root. In this case, there exists a solution $(x,y)\in \ff_{p}^{2}$ 
such that $xy\ne 0$ and thus  Conjecture~\ref{conjB'} holds. 

\begin{prop}
\label{lem2.5}
If $\ff_p$ has a Fibonacci primitive root, then \textup{Conjecture~\ref{conjB'}} holds for this $p$.
\end{prop}

On a related issue, for the valid cases in Conjecture~\ref{conjB'} established above we would like to know how many prime numbers $p$  are there 
such that the subgroup $H$ generated by $-1$ and $2$ has the given index $\ell_{0}$ in $\ff_{p}^{\times}$. 
This question can be viewed as a generalization of the Artin's primitive root conjecture.  It is shown in~\cite[Theorem~1]{Mura} that there are infinitely many primes $p$ such that the index $[\ff_{p}^{\times}: H] = \ell_0 $ under the {\em Generalized Riemann Hypothesis} (GRH).  
Assuming GRH, we conclude from Corollary~\ref{thm1.2} that there are infinitely many prime numbers $p$ with $\ell_{0}$ satisfying conditions in Theorem~\ref{thm:D} and therefore, the size of optimal CACs of prime lengths is equal to $M(p, \ell_{0})$ for {\it infinitely many}  primes $p.$ 

\section*{Acknowledgment} 
We would like to thank Professor Yuan-Hsun Lo for bringing \cite{FuLoShu} to our attention. 
The first named author is partially supported by MOST grant 110-2115-M-003-007-MY2. 
The second named author is partially supported by MOST grant 111-2115-M-003-005. 
The third named author is supported by MOST grant 110-2811-M-003-530.

\bibliographystyle{alphaurl} 
\newcommand{\etalchar}[1]{$^{#1}$}


\end{document}